\documentclass[11pt]{amsart}
\usepackage[T1]{fontenc}

\usepackage{amssymb}
\usepackage{enumerate}
\usepackage{color}
\usepackage{graphicx}
\usepackage{amsmath}
\usepackage{mathtools}

\input xy
\xyoption{all}

\newcommand{\e}{\varepsilon}
\newcommand{\IC}{\mathbb{C}}

\newcommand{\N}{\mathbb{N}}
\newcommand{\R}{\mathbb{R}}
\newcommand{\F}{\mathcal{F}}

\renewcommand{\phi}{\varphi}
\newcommand{\IN}{\mathbb N}
\newcommand{\IR}{\mathbb R}

\newcommand{\w}{\omega}
\newcommand{\0}{\textbf{0}}

\newcommand{\id}{\operatorname{id}}

\newcommand{\Lip}{\operatorname{Lip}}
\newcommand{\diam}{\operatorname{diam}}
\newcommand{\dist}{\operatorname{dist}}
\newcommand{\cl}{\operatorname{cl}}
\newcommand{\Ra}{\Rightarrow}
\newcommand{\defeq}{\coloneqq}

\newtheorem{theorem}{Theorem}[section]
\newtheorem{proposition}[theorem]{Proposition}
\newtheorem{lemma}[theorem]{Lemma}
\newtheorem{claim}[theorem]{Claim}
\newtheorem{corollary}[theorem]{Corollary}
\theoremstyle{definition}
\newtheorem{definition}[theorem]{Definition}
\newtheorem{remark}[theorem]{Remark}
\newtheorem{example}[theorem]{Example}
\usepackage{easyReview}

\title{Characterizing Lipschitz images of injective metric spaces}

\author[T. Banakh, J. B\k{a}k, J. Garbuli\'{n}ska-W\k{e}grzyn, M. Nowak, M. Pop\l{}awski]{Taras Banakh, Judyta B\k{a}k, Joanna Garbuli\'{n}ska-W\k{e}grzyn,\\ Magdalena Nowak, Micha\l{} Pop\l{}awski}

\address{T. Banakh
	(ORCID 0000-0001-6710-4611): 
Ivan Franko National University of Lviv (Ukraine), and 
	Mathematics Department, 
	Jan Kochanowski University in Kielce,
	Uniwersytecka 7, 25-406 Kielce, Poland}
\email{t.o.banakh@gmail.com}

\address{J. B\k{a}k: 
	(ORCID 0000-0001-8027-7226): 
	Mathematics Department, 
	Jan Kochanowski University in Kielce,
	Uniwersytecka 7, 25-406 Kielce, Poland}
\email{judyta.bak@ujk.edu.pl}

\address{J. Garbuli\'nska-W\k{e}grzyn
	(ORCID 0000-0001-7217-2002): 
	Mathematics Department, 
	Jan Kochanowski University in Kielce,
	Uniwersytecka 7, 25-406 Kielce, Poland}
\email{jgarbulinska@ujk.edu.pl}

\address{M. Nowak
	(ORCID 0000-0003-1915-0001): 
	Mathematics Department, 
	Jan Kochanowski University in Kielce,
	Uniwersytecka 7, 25-406 Kielce, Poland}
\email{magdalena.nowak805@gmail.com}

\address{M. Pop\l{}awski
	(ORCID 0000-0002-2725-9675): 
	Mathematics Department, 
	Jan Kochanowski University in Kielce,
	Uniwersytecka 7, 25-406 Kielce, Poland}
\email{michal.poplawski.m@gmail.com}

\date{\today}

\begin{document}
	
\begin{abstract} A metric space $X$ is {\em injective} if every non-expanding map $f:B\to X$ defined on a subspace $B$ of a metric space $A$ can be extended to a non-expanding map $\bar f:A\to X$. We prove that a metric space $X$ is a Lipschitz image of an injective metric space if and only if $X$ is {\em Lipschitz connected} in the sense that for every points $x,y\in X$, there exists a Lipschitz map $f:[0,1]\to X$ such that $f(0)=x$ and $f(1)=y$. In this case the metric space $X$ carries a well-defined intrinsic metric.
A metric space $X$ is a Lipschitz image of a compact injective metric space if and only if $X$ is compact, Lipschitz connected and its intrinsic metric is totally bounded. A metric space $X$ is a Lipschitz image of a separable injective metric space if and only if $X$ is a Lipschitz image of the Urysohn universal metric space if and only if $X$ is analytic, Lipschitz connected and its intrinsic metric is separable.   
\end{abstract}
\maketitle
%

\section{Introduction}

By the classical Hahn--Mazurkiewicz Theorem \cite{Ha,M}, a Hausdorff space $X$ is a continuous image of the unit segment $[0,1]$ if and only if $X$ is a Peano continuum, i.e., $X$ is compact, metrizable, connected and locally connected. 
A metric counterpart of the Hahn--Mazurkiewicz Theorem is due to Wa\.zewski \cite{Waz} and Eilenberg, Harrold \cite{EH} who proved that a metric space $X$ is the image of the closed interval $[0,1]$ under a Lipschitz map if and only if $X$ has finite length. The unit interval and the real line are the most important examples of injective metric spaces. So, it is natural to ask which metric spaces are Lipschitz images of injective metric spaces.
In this paper we address this problem and characterize metric spaces which are Lipschitz images of (separable or totally bounded) injective metric spaces. 
\smallskip

For a metric space $X$, its metric will be denoted by $d_X$. For a map $f:X\to Y$, its {\em Lipschitz constant} is defined as a (finite or infinite) number
$$\Lip(f)\defeq\sup\limits_{{x,x'\in X}\atop{x\ne x'}}\frac{d_Y(f(x),f(x'))}{d_X(x,x')}\in[0,\infty].$$  In this definition we assume that $\sup\varnothing=0$. 

A map $f:X\to Y$ between metric spaces $X,Y$ is called
\begin{itemize}
\item  {\em Lipschitz} if $\Lip(f)<\infty$;
\item {\em $L$-Lipschitz} for a real constant $L$, if $\Lip(f)\le L$;
\item{\em non-expanding} if it is 1-Lipschitz.
\end{itemize}

A metric space $X$ is {\em Lipschitz connected} if for every points $x,y\in X$ there exists a Lipschitz map $f:[0,1]\to X$ such that $f(0)=x$ and $f(1)=y$. Such a map $f$ is called a {\em Lipschitz path} connecting the points $x,y$. 
It is clear that every Lipschitz connected metric space is a path connected topological space. Every Lipschitz connected metric space $X$ carries the {\em intrinsic metric} $d_{IX}:X\times X\to \IR$ defined by
$$
d_{IX}(x,y)\defeq\inf\{\Lip(f):f\in X^{[0,1]}\;\wedge\; f(0)=x\;\wedge\;f(1)=y\}.
$$ 
for any points $x,y\in X$. In the above formula, $X^{[0,1]}$ stands for  the set of all functions from $[0,1]$ to $X$.

A metric space $X$ is called {\em injective} if every $1$-Lipschitz map $f:B\to X$ defined on a subspace $B$ of a metric space $A$ extends to a $1$-Lipschitz map $\bar f:A\to X$. 

Injective metric spaces form a well-known and well-studied class of complete metric spaces, \cite{B}, \cite[p. 47]{GK}. 
 Such spaces are known under two alternative names: hyperconvex metric spaces {\cite{AP, EK, K}} and $1$-Lipschitz Absolute Retracts \cite[Ch.I.1.]{BL}.
The following (known) characterization of injective metric spaces can be deduced from \cite{AP, BL}.

\begin{theorem} For a metric space $X$, the following conditions are equivalent:
\begin{enumerate}
\item[\textup{(1)}]\label{def_inj} $X$ is an injective metric space;
\item[\textup{(2)}]\label{def_LAR} $X$ is {\em $1$-Lipschitz Absolute Retract}, which means that for every metric space $Y$ containing $X$, there exists a $1$-Lipschitz map $f:Y\to X$ such that $f(x)=x$ for all $x\in X$;
\item[\textup{(3)}]\label{def_hyp} $X$ is {\em hyperconvex} in the sense that for any indexed family of closed balls $\big(B[x_i,r_i]\big)_{i\in I}$ in $X$ with $d_X(x_i,x_j)\le r_i+r_j$ for all $i,j\in I$, the intersection $\bigcap_{i\in I}B[x_i,r_i]$ is not empty.
\item[\textup{(4)}]\label{def_m+b}  $X$ is {\em metrically convex} (for every $a,b\in X$ and $t\in[0,1]$ there exists $x\in X$ such that $d_X(a,x)=t\cdot d_X(a,b)$ and $d_X(x,b)=(1-t)\cdot d_X(a,b)$) and {\em has a binary intersection property} (every collection of mutually intersecting closed balls in $X$ has a common point).
\item[\textup{(5)}]\label{def_ext}  for every metric space $Y$ containing $X$ and for every metric space $Z$, every Lipschitz map $X\to Z$ has an extension $Y\to Z$ with the same Lipschitz constant.
\end{enumerate}
\end{theorem}

\begin{proof}
The implications $(5)\Leftrightarrow(1)\Leftrightarrow(2)$ are proved in \cite[Proposition 1.2]{BL}; $(3)\Rightarrow(2)$ in \cite[Theorem 8]{AP}, and $(2)\Leftrightarrow(4)$ in \cite[Proposition 1.4]{BL}. To see that $(4)\Ra(3)$,  take any indexed family $\big(B[x_i,r_i]\big)_{i\in I}$ of closed balls in $X$ such that  $d_X(x_i,x_j)\leq r_i+r_j$ for every $i,j\in I$. The metric convexity of $X$ ensures that for any distinct indices $i,j\in I$, the closed balls $B[x_i,r_i]$ and $B[x_j,r_j]$ have a common point. The binary intersection property guarantees that $\bigcap_{i\in I}B[x_i,r_i]\neq\emptyset$.
\end{proof}

Another important source of injective metric spaces are complete $\IR$-trees, see \cite{C, K}. 
A metric space $X$ is called an {\em $\IR$-tree} if for every points $x,y\in X$ the {\em metric segment} $$[x,y]\defeq\{z\in X:d_X(x,y)=d_X(x,z)+d_X(z,y)\}$$is isometric to the interval $[0,d_X(x,y)]$ on the real line, and for every points $x,y,z\in X$ with $[x,y]\cap[y,z]=\{y\}$ we have $d_X(x,z)=d_X(x,y)+d_X(y,z)$. By \cite{K}, every complete $\IR$-tree is an injective metric space. Nevertheless, classes of $\IR$-trees and injective metric spaces are distinct. For instance $(0,1)$ (with the euclidean metric) is $\IR$-tree, which is not injective, since injective metric spaces are complete. On the other hand, $[0,1]^2$ with the maximum metric is an injective metric space, which is not an $\IR$-tree (the metric segment $[(0,0),(1,0)]$ is not isometric to $[0,1]$). 

A metric space $Y$ is called a {\em Lipschitz image} (resp. {\em $1$-Lipschitz image}) of a metric space $X$ if there exists a surjective Lipschitz map $f:X\to Y$ (with $\Lip(f)\le 1$).

One can easily reformulate Hahn-Mazurkiewicz theorem \cite{Ha,M}  to the following version$\colon$ a Hausdorff space $X$ is a continuous image of the unit segment $[0,1]$ if and only if $X$ is a continuous image of some compact injective metric space $Y$. 
However, the Ważewski result in \cite{Waz} implies that the class of Lipschitz images of the unit interval $[0,1]$ is strictly smaller than the class of Lipschitz images of compact injective metric spaces: a counterexample is the injective metric space $[0,1]^2$ which is not the Lipschitz image of the unit segment $[0,1]$. 

The following three characterizations proved in Sections~\ref{s:hedgehog}, \ref{s:compact}, \ref{s:separable} are the main results of this paper.

\begin{theorem}\label{t:hedgehog} For a metric space $X$, the following conditions are equivalent:
\begin{enumerate}
\item[\textup{1)}] $X$ is a Lipschitz image of an injective metric space;
\item[\textup{2)}] $X$ is a Lipschitz image of a complete $\IR$-tree;
\item[\textup{3)}] $X$ is Lipschitz connected.
\end{enumerate}
\end{theorem}


\begin{theorem}\label{t:compact} For a metric space $X$, the following conditions are equivalent: 
\begin{enumerate}
\item[\textup{1)}] $X$ is a Lipschitz image of a totally bounded injective metric space;
\item[\textup{2)}] $X$ is a Lipschitz image of  a compact $\IR$-tree;
\item[\textup{3)}] $X$ is compact, Lipschitz connected, and its intrinsic metric is totally bounded.
\end{enumerate}
\end{theorem}

\begin{theorem}\label{t:analytic} For a metric space $X$, the following conditions are equivalent:
\begin{enumerate}
\item[\textup{1)}] $X$ is a Lipschitz image of a separable injective metric space;
\item[\textup{2)}] $X$ is a Lipschitz image of a separable complete $\IR$-tree;
\item[\textup{3)}] $X$ is a $1$-Lipschitz image of the Urysohn universal space;
\item[\textup{4)}] $X$ is analytic, Lipschitz connected, and its intrinsic metric is separable.
\end{enumerate}
\end{theorem}

Let us recall that a topological space $X$ is {\em analytic} if it is a continuous image of a separable complete metric space. 
The {\em Urysohn universal metric space} is a complete separable metric space $\mathbb{U}$ such that any isometric embedding of a subset $A$ of a finite metric space $B$ admits an extension to an isometric embedding of $B$ in $\mathbb{U}$. It is well-known \cite{Urys} that a Urysohn universal metric space exists and is unique up to an isometry.

\section{$\IR$-convex metric spaces and $\IR$-trees}

\begin{definition} A metric space $X$ is called {\em $\IR$-convex} if for every points $x,z\in X$, the metric segment
$$[x,z]\defeq\{y\in X:d_X(x,y)+d_X(y,z)=d_X(x,z)\}$$is isometric to the segment $[0,d_X(x,z)]$ of the real line.
\end{definition}

\begin{proposition}\label{p:unicon} For a metric space $X$, the following conditions are equivalent:
\begin{enumerate}
\item[\textup{1)}] $X$ is $\IR$-convex;
\item[\textup{2)}] for every points $x,y\in X$, there exists a unique set $I\subseteq X$ such that $x,y\in I$ and $I$ is isometric to the segment $[0,d_X(x,y)]$ on the real line;
\item[\textup{3)}] for every points $x,y\in X$ there exists a unique $1$-Lipschitz map $f:[0,d_X(x,y)]\to X$ such that $f(0)=x$ and $f(d_X(x,y))=y$.
\end{enumerate}
\end{proposition} 

\begin{proof} The implications $(1)\Ra(2)\Ra(3)$ are obvious. To prove that $(3)\Ra(1)$, assume that a metric space $X$ satisfies the condition (3). Given any points $x,y\in X$, we need to show that the metric segment $[x,y]$ is isometric to the interval $[0,d_X(x,y)]$ in the real line. By the condition (3), there exists a unique $1$-Lipschitz map $f:[0,d_X(x,y)]\to X$ such that $f(0)=x$ and $f(d_X(x,y))=y$. The triangle inequality implies that the image $I\defeq f\big[[0,d_X(x,y)]\big]$ is a subset of the metric segment $[x,y]$. Assuming that $[x,y]\ne I$, we can find a point $z\in [x,y]\setminus I$. Then $d_X(x,y)=d_X(x,z)+d_X(z,y)$. By condition (3), there exist $1$-Lipschitz maps $\alpha:[0,d_X(x,z)]\to X$ and $\beta:[0,d_X(z,y)]\to X$ such that $\alpha(0)=x$, $\alpha(d_X(x,z))=z=\beta(0)$ and $\beta(d_X(z,y))=y$. The  $1$-Lipschitz maps $\alpha,\beta$ compose a $1$-Lipschitz map $\gamma:[0,d_X(x,y)]\to X$ defined by
$$\gamma(t)=\begin{cases}
\alpha(t)&\mbox{if $t\in[0,d_X(x,z)]$};\\
\beta(t-d_X(x,z))&\mbox{if $t\in[d_X(x,z),d_X(x,y)]$}.
\end{cases}
$$
Since $\gamma(0)=x=f(0)$, $\gamma(d_X(x,y))=y=f(d_X(x,y))$, and $\gamma\ne f$, we obtain a contradiction with the uniqueness of the $1$-Lipschitz path $f$. This contradiction shows that $[x,y]=I$. The triangle inequality implies that the $1$-Lipschitz map $f:[0,d_X(x,y)]\to [x,y]$ is an isometry, witnessing that the metric space $X$ is $\IR$-convex. Indeed, for every $u,z$ such that $0 \leq u\leq z \leq d_X(x,y)$ we have $|u-z| \geq d_X(f(u),f(z)) \geq d_X(x,y)-d_X(x,f(u))-d_X(f(z),y) \geq d_X(x,y)-|0-u|-|z-d_X(x,y)|=|u-z|.$
\end{proof}

\begin{proposition} If a metric space $X$ is $\IR$-convex, then for every points $x,y,z\in X$, the intersection $[x,y]\cap [x,z]$ is equal to the metric segment $[x,u]$ for a unique point $u\in[x,y]\cap[x,z]$.
\end{proposition}

\begin{proof} Since the metric space $X$ is $\IR$-convex, there exist isometries $\alpha:[0,d_X(x,y)]\to[x,y]$ and $\beta:[0,d_X(x,z)]\to [x,z]$ such that $\alpha(0)=x=\beta(0)$, $\alpha(d_X(x,y))=y$, and $\beta(d_X(x,z))=z$. By the compactness of $[x,z]$, the closed set $\alpha^{-1}[x,z]$ in $[0,d_X(x,y)]$ has the largest element $s$. Since $\alpha$ is an isometry, $s=d_X(\alpha(0),\alpha(s))=d_X(x,\alpha(s))$. Let $u=\alpha(s)$ and  $t=\beta^{-1}(u)$. Since $\beta$ is an isometry, $t=d_X(\beta(0),\beta(t))=d_X(x,u)=s$. Observe that $\alpha'\defeq \alpha{\restriction}_{[0,s]}:[0,s]\to X$ and $\beta'\defeq \beta{\restriction}_{[0,t]}:[0,t]\to X$ are two $1$-Lipschitz maps  from the interval $[0,s]=[0,t]$ to $X$ such that such that $\alpha'(0)=x=\beta'(0)$ and $\alpha'(s)=u=\beta'(t)$.  Proposition~\ref{p:unicon} ensures that $\alpha'=\beta'$ and hence $[x,y]\cap[x,z]=\alpha[[0,s]]=[x,u]$. The uniqueness of the point $u$ implies from the uniqueness of the end-points of the segment $[x,u]=[x,y]\cap[x,z]$.
\end{proof}

We observe that a metric space $X$ is $\IR$-tree if $X$ is an $\IR$-convex metric space such that for every points $x,y,z\in X$ with $[x,y]\cap[y,z]=\{y\}$, we have $d_X(x,z)=d_X(x,y)+d_X(y,z)$.

The following important fact can be found in \cite{K}. 

\begin{proposition}\label{Rtree} Every complete $\IR$-tree is an injective metric space.
\end{proposition}

\section{Lipschitz connected metric spaces}

In this section we study Lipschitz connected and $L$-Lipschitz connected {metric} spaces in more details.

\begin{definition} Let $L$ be a real number. A metric space $X$ is called {\em $L$-Lipschitz connected} if for every points $x,y\in X$, there exists an $L$-Lipschitz path $f:[0,d_X(x,y)]\to X$ such that $f(0)=x$ and $f(d_X(x,y))=y$.
\end{definition}

If an $L$-Lipschitz connected metric space contains more than one point, then $L\ge 1$. Therefore, the $1$-Lipschitz connectedness is the strongest notion among the $L$-Lipschitz connectedness properties. 

\begin{example}\label{ex:injective=>Lc} Every injective metric space $X$ is $1$-Lipschitz connected.
\end{example}

\begin{proof} Given any point $x,y\in X$, consider the doubleton $\{0,d_X(x,y)\}$ in the real line and the $1$-Lipschitz map $f:\{0,d_X(x,y)\}\to X$ defined by $f(0)=x$ and $f(d_X(x,y))=y$. By the injectivity of $X$, this $1$-Lipschitz map admits a $1$-Lipschitz extension $\bar f:[0,d_X(x,y)]\to X$, which is a $1$-Lipschitz path witnessing that $X$ is $1$-Lipschitz connected.
\end{proof}




As we have mentioned in the introduction, every Lipschitz connected metric space $X$ carries the intrinsic metric 
$d_{IX}:X\times X\to \IR$ defined by 
$$
\begin{aligned}
d_{IX}(x,y)&=\inf\{\Lip(f):f\in X^{[0,1]}\;\wedge\; f(0)=x\;\wedge\;f(1)=y\}\\
&=\inf\{C\in[0,\infty):\exists f\in X^{[0,C]}\;(f(0)=x\;\wedge\;f(C)=y\;\wedge\;\Lip(f)\le 1)\}
\end{aligned}
$$ 
for any points $x,y\in X$. The second equality in the definition of the intrinsic metric holds because a function $f\in X^{[0,1]}$ is $L$-Lipschitz if and only if the function $f_L:[0,L]\to X$, $f_L:t\mapsto f(t/L)$, is $1$-Lipschitz.

The definition of the intrinsic metric implies that $d_X\le d_{IX}$ for every Lipschitz connected metric space $X$. 
Also any $1$-Lipschitz map $f:X\to Y$ between Lipschitz connected metric spaces remains $1$-Lipschitz with respect to the intrinsic metrics on $X$ and $Y$, which means that $$d_{IY}(f(x),f(x'))\le d_{IX}(x,x')$$ for every points $x,x'\in X$. 

The following lemma describes an interplay between the completions of a Lipschitz connected metric space with respect to the intrinsic metric and the original metric.

\begin{lemma}\label{l:preimage} Let $X$ be a Lipschitz connected metric space, $\overline X$ be the completion of $X$, $\overline X_I$ be the completion of the metric space $X_I\defeq(X,d_{IX})$, and $\overline \id:\overline X_I\to \overline X$ be the unique continuous extension of the identity map $\id:X_I\to X$. Then $\overline \id^{-1}[X]=X_I$. In other words, every Cauchy sequence $(x_n)_{n \in \N}$ in $(X,d_{IX})$, which is convergent to a point $x$ in $(X,d_X)$ is convergent to $x$ in $(X,d_{IX}),$ too.
\end{lemma}

\begin{proof} Given any point $x\in \overline X_I$ with $\overline \id(x)\in X$, we need to prove that $x\in X_I$. Since $x\in\overline X_I$, there exists a sequence $(x_n)_{n\in\N}$ in $X_I$ such that $d_{IX}(x,x_n)<\frac1{2^{n+2}}$ for all $n\in\N$. For every $n\in\N$ we have $$d_{IX}(x_n,x_{n+1})\le d_{IX}(x_n,x)+d_{IX}(x,x_{n+1})<\tfrac1{2^{n+2}}+\tfrac1{2^{n+3}}<\tfrac1{2^{n+1}}.$$ So, we can find a 1-Lipschitz path $\gamma_n:[\tfrac1{2^{n+1}},\tfrac1{2^n}]\to X$ such that $\gamma_n(\tfrac1{2^{n+1}})=x_{n+1}$ and $\gamma_n(\tfrac1{2^n})=x_n$. The paths $\gamma_n$, $n\in\N$, compose a 1-Lipschitz path $\gamma:[0,1]\to X$ defined by $\gamma(0)=\overline \id(x)$ and $\gamma{\restriction}_{[\tfrac1{2^{n+1}},\tfrac1{2^n}]}=\gamma_n$ for all $n\in\N$. The definition of $\gamma$ guarantees that $$\lim_{n\to\infty}d_{IX}(x,x_n)=0=\lim_{n\to\infty}d_{IX}(x_n,\id^{-1}(\overline \id(x)))$$ and hence $x=\id^{-1}(\overline \id(x))\in \id^{-1}[X]=X_I$.
\end{proof}

A metric $d$ on a set $X$ will be called {\em compact} (resp. {\em complete}, {\em separable}, {\em analytic}) if so is the metric space $(X,d)$.

\begin{lemma} \label{d vs dp} Let $(X,d_X)$ be a Lipschitz connected metric space.
\begin{enumerate}
\item[\textup{(1)}] \label{trudne} If the metric $d_X$ is complete, then so is the intrinsic metric $d_{IX}$.
\item[\textup{(2)}]	The metric $d_{IX}$ is analytic if and only if $d_{IX}$ is separable and the metric $d_X$ is analytic.
\item[\textup{(3)}] If the metric $d_{IX}$ is separable (resp. totally bounded, compact), then so is the metric $d_X$.
\end{enumerate}
\end{lemma}

\begin{proof} (1) Suppose that the metric $d_X$ is complete and take a Cauchy sequence $(x_n)_{n \in \N}$ in the metric space $(X,d_{IX}).$ Since $d_X\le d_{IX}$, the sequence  $(x_n)_{n \in \N}$ remains Cauchy in the metric space $(X,d_{X})$. Since $(X,d_X)$ is complete,  this sequence converges to some point $x \in X$ in the metric space $(X,d_X)$. Lemma \ref{l:preimage} guarantees that $\lim_{n \to \infty} x_n=x$ in $(X,d_{IX})$ too, witnessing that the intrinsic metric $d_{IX}$ is complete.
\smallskip

(2) If the metric $d_{IX}$ is analytic, then $d_{IX}$ is separable and $d_X$ is analytic, because of the continuity of the identity map $(X,d_{IX})\to (X,d_X)$.

Next, assume that the metric $d_{IX}$ is separable and $d_X$ is analytic. Let $\overline X$ be the completion of the metric space $(X,d_X)$ and $\overline X_I$ be the completion of the metric space $X_I\defeq(X,d_{IX})$. Let $\bar I:\overline X_I\to\overline X$ be the continuous extension of the (1-Lipschitz) identity map $I:X_I\to X$. By Lemma~\ref{l:preimage}, $X_I=\bar I^{-1}[X]$. Since the spaces $\overline X_I$ and $\overline X$ are Polish, the preimage $X_I=\bar I^{-1}[X]$ is an analytic subspace of $\overline X_I$, by \cite[14.4]{Ke}.
\smallskip

The statement (3) follows from the $1$-Lipschitz property of the identity map $(X,d_{IX})\to(X,d_{IX})$.
\end{proof}

The following simple example shows that the implications of Lemma \ref{d vs dp} cannot be reversed in a general case.

	\begin{example} \label{e: d vs dp} Let $L$ be any line in the complex plane $\IC$ and $o\in\mathbb C\setminus L$ be any point. For every uncountable compact zero-dimensional set $K\subseteq L$, the cone $$X:=\bigcup_{x\in K}[o,x]$$over $K$ with vertex $o$ is a Lipschitz connected compact (by \cite[3.2.11]{E}) metric space with respect to the metric $d_X$ inherited from the complex plane. Thus $(X,d_X)$ is also separable and totally bounded. Since $d_{IX}(x,y)\ge 2\cdot \dist(o,L)$  for any distinct points $x,y\in K$, the metric space $(X,d_{IX})$ is not separable (and hence not totally bounded and not compact). 
		
	If the compact space $K$ has no isolated points, then for every countable dense set $D\subseteq K$, the Lipschitz connected metric subspace $Y=\bigcup_{x\in D}[o,x]$ of $X$ is not complete (by the Baire Theorem) whereas the intrinsic metric $d_{IY}$ is $\sigma$-compact and complete.
	\end{example}
	

\begin{example} The metric subspace $X=\{(x,\sin(\frac{1}{x})) \colon x>0\}$ of the Euclidean plane is Lipschitz connected and not complete whereas the intrinsic metric on $X$ is complete.
\end{example}

	

\section{The $\IR$-trees generated by weighted trees}

In this section we describe a general construction of a (complete) $\IR$-tree, determined by a weighted tree.

\begin{definition} A {\em forest} is a partially ordered set $(T,\preceq)$  such that for every $t\in T$ the set $${\downarrow}t\defeq \{x\in T:x\preceq t\}$$ is finite and linearly ordered. 
\end{definition}

 For any elements $s,t\in T$, the largest element of the linearly ordered set ${\downarrow}s\cap{\downarrow}t$ is called the {\em infimum} of $s$ and $t$ and is denoted by $s\wedge t$.

For an element $t$ of a forest $T$, the set $\mathrm{succ}(t)\defeq\min \{s\in T:t\prec s\}$ is called the set of {\em immediate successors} of $t$ in the forest. 

A subset $B$ of a forest $T$ is called a {\em branch} of $T$ if $B$ is a maximal linearly ordered subset of $T$. 
The {\em boundary} $\partial T$ of a forest $T$ is the set of all infinite branches in $T$. 

For a forest $T$ by $\min T$ denote the set of minimal elements of $T$. For a nonempty forest $T$, the set $\min T$ is not empty. A  forest $T$ is a {\em tree} if the set $\min T$ is a singleton, consisting of the smallest element of $T$, called the {\em root} of the tree $T$ and denoted by $r_T$.

Observe that for every element $r$ of a forest $T$, the set 
$${\uparrow}r\defeq\{x\in T:r\preceq x\}$$ is a tree with root $r$.

\begin{definition} A {\em  weighted tree} is a pair $(T, w)$ consisting of a tree $T$ 
and a function $w:T\to[0,+\infty)$ assigning to every vertex $t\in T$ its {\em weight} $w_t=w(t)$ so that $w_t=0$ if and only if $t$ is the root of $T$.
\end{definition}

Every weighted tree $(T,w)$ determines an $\IR$-tree  $$T_w=(\{r_T\}\times\{0\})\cup\bigcup_{t\in T}\{t\}\times (0,w_t],$$endowed with the  metric
$$d_w((t,t'),(s,s')):=\begin{cases}\;|t'-s'|&\mbox{if $s=t$};\\
t'-w_t+s'-w_s+\sum_{x\in {\downarrow}s\triangle {\downarrow}t}w_x&\mbox{otherwise},\\
\end{cases}$$
where $A\triangle B$ stands for the symmetric difference $(A\setminus B)\cup(B\setminus A)$ of the sets $A,B$. It is easy to check that the metric space $(T_w,d_w)$ is indeed an $\IR$-tree. The $\IR$-tree $T_w$ will be called the {\em $\IR$-tree of the weighted tree $(T,w)$}.

The $\IR$-tree $T_w\subset T\times \R$ carries a natural  lexicographic order
	$$(s,s') \leq_L (t,t')\quad\Leftrightarrow \quad s\prec t \;\vee\; (s=t \;\wedge\; s'\leq t').$$ So, for every element $t''\in T_w$ we can consider its  lower and upper sets
$${\downarrow}t''\defeq\{s''\in T_w:s''\le_L t''\}\quad\mbox{and}\quad{\uparrow}t''\defeq\{s''\in T_w:t''\le_L s''\}$$
in the partially ordered set $(T_w,\le_L)$. Since $T$ is a tree,  the lower set ${\downarrow}t''$ of every element $t''\in T_w$ is linearly ordered. For two elements $a'',b''\in T_w$, let $a''\wedge b''$ denote the largest element of the linearly ordered set ${\downarrow}a''\cap{\downarrow}b''$ in the partially ordered set $(T_w,\le_L)$. 

Now we shall characterize weighted trees whose $\IR$-trees are totally bounded or separable.

\begin{definition} A weighted tree $(T,w)$ is defined to be {\em precompact} if 
for every $\e>0$ there exists a finite set $F\subseteq T$ such that for every finite linearly ordered set $L\subseteq T\setminus F$ we have $\sum_{x\in L}w_x<\e$.
\end{definition}

\begin{proposition}\label{calzup} A weighted tree $(T,w)$ is precompact if and only if the metric space $T_w$ is totally bounded.
\end{proposition}

\begin{proof} To prove the ``if'' part, assume that the metric space $T_w$ is totally bounded. Then for every $\varepsilon>0$, there exists a finite set $F''\subseteq T_w$ such that for every $x''\in T_w$, there exists $y''\in F''$ with $d_w(x'',y'')<\e$. Since $F''$ is finite, there exists a finite set $F\subseteq T$ such that  $F''\subseteq\bigcup\{{\downarrow}(t,w_t):t\in F\}\subseteq T_w$. Replacing $F$ by the finite set ${\downarrow}F\defeq\bigcup_{t\in F}{\downarrow}t$, we can assume that $F={\downarrow}F$. We claim that $\sum_{x\in L}w_x<\e$ for every finite linearly ordered set $L\subseteq T\setminus F$. Since $L$ is finite and linearly ordered, there exists an element $u\in L$ such that $L\subseteq{\downarrow}u$. By the choice of the set $F''$, there exists a point $y''\in F''$ such that $d_w((u,w_u),y'')<\e$. Write the element $y''$ as $(y,y')$ for some $y\in T$ and $y'\in [0,w_y]$. The choice of the set $F$ ensures that $y\in F$ and hence $L\subseteq {\downarrow}u\setminus F\subseteq {\downarrow}u\setminus{\downarrow}y$. Then $$\sum_{x\in L}w_x\le \sum_{x\in {\downarrow}u\setminus{\downarrow}y}w_x\le d_w((u,w_u),y'')< \varepsilon,$$ witnessing that the weighted tree $(T,w)$ is precompact.
\smallskip

\begin{picture}(100,100)(-200,-20)
\put(0,0){\line(1,1){60}}
\put(0,0){\line(-1,1){30}}
\put(-45,45){\circle*{2}}
\put(-50,50){\circle*{2}}
\put(-55,35){$L$}
\put(-40,40){\circle*{2}}
\put(0,0){\circle*{3}}

\put(30,30){\circle*{3}}
\put(60,60){\circle*{3}}
\put(-97,58){$(u,w_u)$}
\put(-30,30){\circle*{3}}
\put(63,58){$(y,w_y)$}
\put(-60,60){\circle*{3}}
\put(40,40){\circle*{2}}
\put(43,35){$y''=(y,y')\in F''$}
\put(-30,-10){$(u\wedge y, w_{u\wedge y})$}
\end{picture}

To prove the ``only if'' part, assume that the weighted tree $(T,w)$ is precompact. To prove that the metric space $T_w$ is totally bounded, fix any $\e>0$. By the precompactness of the weighted tree $(T,w)$, there exists a nonempty finite set $F\subseteq T$ such that $\sum_{x\in L}w_x<\e$ for every finite linearly ordered set $L\subseteq T\setminus F$. Replacing $F$ with the finite set ${\downarrow}F\defeq\bigcup_{t\in F}{\downarrow}t$, we can assume that $F={\downarrow}F$. Since the set $F$ is finite, the set $$K\defeq{\textstyle\bigcup}\{{\downarrow}(t,w_t):t\in F\}=(\{r_T\}\times\{0\})\cup\bigcup_{t\in F}(\{t\}\times(0,w_t])$$is compact and hence totally bounded in the metric space $T_w$. Then there exists a finite set $F''\subseteq K$ such that for every $x''\in K$ there exists $y''\in F''$ with $d_w(x'',y'')<\e$. Replacing $F''$ by the set $F''\cup \{(t,w_t):t\in F\}$, we can assume that $\{(t,w_t):t\in F\}\subseteq F''$. We claim that for every $x''\in T_w$, there exists $y''\in F''$ such that $d_w(x'',y'')<\e$.

Indeed, if $x''\in K$, then the existence of a point $y''\in F''$ with $d_w(x'',y'')<\e$ follows from the choice of the set $F''$. So, we assume that $x''\notin K$.
In this case, $x''=(x,x')$ for some $x\in T\setminus F$ and $x'\in (0,w_x]$. Since $r_T\in {\downarrow}F=F$, there exists a point $y\in F\cap{\downarrow}x$ such that $F\cap{\downarrow}x={\downarrow}y$. The choice of the set $F''$ ensures that it contains the point $y''\defeq (y,w_y)$. The choice of $F$ guarantees that $\sum_{t\in{\downarrow}x\setminus {\downarrow}y}w_t=\sum_{t\in{\downarrow}x\setminus F}w_t<\e$ and hence  
$d_w(x'',y'')=x'-w_x+\sum_{t\in {\downarrow}x\setminus {\downarrow}y}w_t\le \sum_{t\in {\downarrow}x\setminus {\downarrow}y}w_t<\e,$ by the definition of the metric $d_w$. This shows that the metric space $T_w$ is totally bounded.
\end{proof}


\begin{proposition}\label{przel} For a weighted tree $(T,w)$, the $\IR$-tree $T_w$ is separable if and only if the tree $T$ is countable.
\end{proposition}

\begin{proof} If the tree $T$ is countable, then $T_w\cap(T\times\mathbb Q)$ is a countable dense set in the $\IR$-tree, witnessing that the metric space $T_w$ is separable. 

Next, assume that the tree $T$ is uncountable. The definition of metric $d_w$ ensures that for every node $t\in T$, the set $\{t\} \times (0,w_t)$ is open in the metric space $T_w$. Therefore, the metric space $T_w$ contains the uncountable family of pairwise disjoint open sets $\big\{\{t\}\times(0,w_t):t\in T\big\}$ and hence $T_w$ cannot be separable, according to \cite[Corollary 4.1.16]{E}. 
\end{proof}

Let $(T,w)$ be a weighted tree. In the set $\partial T$ of all infinite branches of $T$, consider the subset 
$$\partial T_w:=\{L\in\partial T:|L|=\infty,\;\sum_{t\in L}w_t<\infty\}.$$
On the set $\overline T_w:=T_w\cup\partial T_w$ we define a metric $\bar d_w$,
that extends the metric $d_w$ on $T_w$ as follows:
$$\bar d_w(x'',y''):=\begin{cases}\;
|x'-y'|&\mbox{if $x''=(x,x'), y''=(y,y')\in T_w$ and $x=y$};\\
x'-w_x+y'-w_y+\sum_{t\in {\downarrow}x\triangle {\downarrow}y} w_t&\mbox{if $x''=(x,x'), y''=(y,y')\in T_w$ and $x\ne y$};\\
x'-w_x+\sum_{t\in {\downarrow}x\triangle y''} w_t&\mbox{if $x''=(x,x')\in T_w$, $y''\in \partial T_w$ and $x\notin y''$ } ;\\
{w_x}-x'+\sum_{t\in {y''\setminus{\downarrow}x}} w_t&\mbox{if $x''=(x,x')\in T_w$, $ y''\in \partial T_w$ and $x\in y''$ } ;\\
\;\sum_{t\in x''\triangle y''} w_t & \mbox{if $x'',y''\in \partial T_w$}.
\end{cases}$$



\begin{proposition}\label{zup}
For every weighted tree $(T,w)$, the metric space $\overline T_w$ is a completion of the metric space $T_w$.
\end{proposition}

\begin{proof} The definition of the metric on $\overline T_w$ implies that the metric space $T_w$ is a dense subspace of the metric space $\overline T_w$.
To {prove} that $\overline T_w$ is a completion of $T_w$, it suffices to show that every Cauchy sequence $(s''_n)_{n\in\w}$ in the metric space $T_w$ has a limit in the space $\overline T_w$. If for some element $t$ of $T_w$, the set $\{n\in\w:s''_n=t\}$ is infinite, then $t$ is the limit of the Cauchy sequence $(s''_n)_{n\in\w}$ and we are done. So, we can assume that all elements of the Cauchy sequence $(s''_n)_{n\in\w}$ are pairwise distinct and moreover all of them are distinct from the root $(r_T,0)$ of the $\IR$-tree $T_w$. For every $n\in\w$, write the element $s''_n$ of $T_w$ as a pair $(s_n,s_n')$ where $s_n\in T$ and $s_n'\in(0,w_{s_n}]$. Let $\check s_n\defeq \max({\downarrow}s_n\setminus\{s_n\})$ be the immediate predecessor of $s_n$ in the tree $T$.

In the tree $T$, consider the set 
$$\Lambda\defeq\{t\in T:|\{n\in\w:t\preceq\check s_n\}|=\w\}$$and observe that it contains the root $r_T$ of the tree $T$ and hence is not empty.  
Two cases are possible.
\smallskip

1. First assume that the set $\Lambda$ has no maximal elements. In this case $\Lambda$ contains an increasing sequence $(\lambda_n)_{n\in\w}$. For this sequence the set $b\defeq \bigcup_{n\in\w}{\downarrow}\lambda_n$ is an infinite branch of the tree $T$ such that for every $t\in b$ the set $\{n\in\w:t\preceq \check s_n\}$ is infinite. Write the branch $b$ as $b=\{b_n\}_{n\in\w}$, where $b_n$ is a unique element of $b$ such that $|{\downarrow}b_n|=n+1$. We claim that $b\in\partial T_w$. Since the sequence $(s''_n)_{n\in\w}$ is Cauchy, for every  $\e>0$, there exists $m\in\IN$ such that $d_w(s''_n,s''_m)<\e$ for every $n\ge m$. 

Since the sequence $(b_n)_{n\in\w}$ is unbounded in the tree $T$, there exists a number $i\in\w$ such that $b_i\not\preceq \check s_m$. Given any number $j\ge i$, find a number $n\ge m$ such that $b_j\preceq \check s_n$.
Observe that $b_i,b_j,\check s_n$ and $
\check s_n\wedge\check s_m$ are elements of the linearly ordered set ${\downarrow}\check s_n$. Assuming that $\check s_m\wedge\check s_n\not\preceq b_i$, we conclude that $b_i\preceq\check s_m\wedge\check s_n\preceq \check s_m$, which contradicts the choice of the number $i$. This contradiction shows that $\check s_m\wedge\check s_n\preceq b_i$.

In this case, the definition of the metric $d_w$ ensures that $$\sum_{k=i+1}^jw_{b_k}\le-w_{s_n}-w_{s_m}+\sum_{x\in{\downarrow}s_n\triangle{\downarrow}s_m}w_x\le d_w(s''_m,s''_n)<\e$$ and hence $\sum_{k=i+1}^\infty w_{b_k}\le\e<\infty$, witnessing that $b\in\partial T_w$. 

\begin{picture}(100,180)(-200,-20)
\put(0,0){\line(0,1){50}}
\put(0,0){\line(1,1){50}}
\put(0,60){\circle*{1}}
\put(0,65){\circle*{1}}
\put(0,70){\circle*{1}}
\put(0,80){\line(0,1){40}}

\put(0,90){\line(-1,1){50}}
\put(-25,115){\circle*{3}}
\put(-33,107){$\check s_n$}
\put(-50,140){\circle*{3}}
\put(-58,132){$s_n$}

\put(0,90){\circle*{3}}
\put(3,87){$b_j$}
\put(0,110){\circle*{3}}
\put(3,107){$b_{j+1}$}

\put(0,130){\circle*{1}}
\put(0,135){\circle*{1}}
\put(0,140){\circle*{1}}

\put(0,0){\circle*{3}}
\put(3,-6){$\check s_m\wedge\check s_n$}
\put(0,20){\circle*{3}}
\put(3,17){$b_i$}
\put(0,40){\circle*{3}}
\put(3,37){$b_{i+1}$}
\put(25,25){\circle*{3}}
\put(28,19){$\check s_m$}
\put(50,50){\circle*{3}}
\put(53,44){$s_m$}
\end{picture}

Since the sequence $(s''_n)_{n\in\w}$ is Cauchy and $\e>0$ was arbitrary, the equality $b=\lim_{n\to\infty}s''_n$ will follow as soon as we check that  $\overline d_w(b,s''_n)<2\e$ for infinitely many numbers $n$. Indeed, for every number $n\ge m$ in the infinite set $\{n\in\w:b_i\preceq \check s_n\}$ we have $\check s_n\wedge \check s_m\preceq b_i$ and
$$
\begin{aligned}
\overline d_w(s''_n,b)&=s_n'-w_{s_n}+\sum_{x\in {\downarrow}s_n\triangle b}w_x=s_n'-w_{s_n}+\sum_{x\in{\downarrow}s_n\setminus b}w_x+\sum_{x\in b\setminus{\downarrow}s_n}w_x\\
& \le s'_n-w_{s_n}+\sum_{x\in{\downarrow}s_n\setminus{\downarrow}b_i}w_x+{\sum_{k=i+1}^\infty w_{b_k}}\\
&\le s'_n-w_{s_n}+\sum_{x\in{\downarrow}s_n\setminus{\downarrow}s_m}w_x+\sum_{k=i+1}^\infty w_{b_k}\le d_w(s''_n,s''_m)+\e<2\e.
\end{aligned}
$$

2. Next, assume that the set $\Lambda$ has a maximal element $\lambda$. Consider the infinite set $I\defeq\{n\in\w:\lambda\preceq \check s_n\}$. If for some $s\in T$, the set $J\defeq\{n\in I:s=s_n\}$ is infinite, then the maximality of $\lambda$ ensures that $\lambda$ is the immediate predecessor of $s$ in the tree $T$. Then the subsequence $(s_n'')_{n\in J}$ of the sequence $(s''_n)_{n\in\w}$ is contained in the compact set $\{(\lambda,w_\lambda)\}\cup(\{s\}\times(0,w_s])$ and hence the Cauchy sequence $(s_n'')_{n\in\w}$ converges to some points of this compact set. So, assume that for every $s\in T$, the set $\{n\in I:s_n=s\}$ is finite. 
In this case we shall show that the Cauchy sequence $(s_n'')_{n\in\w}$ converges to the element $\lambda''\defeq (\lambda,w_\lambda)$ of the $\IR$-tree $T_w$. Since the sequence $(s_n'')_{n\in\w}$ is Cauchy and consists of pairwise distinct elements, it suffices to show that for every $\e>0$ there exists $n\in\w$ such that $d_w(\lambda'',s_n'')<\e$. Since the sequence $(s''_n)_{n\in\w}$ is Cauchy, there exists $m\in\w$ such that $d_w(s_n'',s_k'')<\e$ for all $n,k\ge m$. We can assume that $\lambda\preceq \check s_m\prec s_m$. Thus the element $\lambda$ has an immediate successor $\hat\lambda$ in the tree $T$ such that $\lambda\prec\hat\lambda\preceq s_m$. The maximality of $\lambda$ ensures that $\hat\lambda\notin\Lambda$ and hence the set $\{n\in\w:\hat \lambda\preceq \check s_n\}$ is finite. Then we can choose a number $n\in I$ such that $n\ge m$ and $\hat\lambda\not\preceq\check s_n$. Then $s_n\wedge s_m=\check s_n=\lambda$ and 
$$d_w(\lambda'',s''_n)\le s'_n-w_{s_n}+s'_m-w_{s_m}+\sum_{x\in{\downarrow}s_n\triangle{\downarrow}s_m}w_x=d_w(s_n'',s_m'')<\e.$$

\begin{picture}(100,100)(-200,-20)
\put(0,0){\line(1,1){60}}
\put(0,0){\line(-1,1){30}}

\put(0,0){\circle*{3}}

\put(-14,-10){$\hat s_n{=}\lambda$}
\put(30,30){\circle*{3}}
\put(33,22){$\hat\lambda$}
\put(60,60){\circle*{3}}
\put(63,55){$s_m$}
\put(-30,30){\circle*{3}}
\put(-38,22){$s_n$}
\end{picture}
\end{proof}



Propositions~\ref{Rtree}, \ref{calzup}, \ref{przel}, \ref{zup} and the definition of the metric $\overline d_w$ imply the following corollary describing the property of the completion $\overline T_w$ of the $\IR$-tree $T_w$ of a weighted tree $(T,w)$.

\begin{corollary} \label{crtree} For every weighted tree $(T,w)$,
\begin{enumerate}
\item[\textup{(1)}] the complete metric space $\overline{T}_w$ is an $\IR$-tree and hence is an injective metric space;
\item[\textup{(2)}] the space $\overline T_w$ is compact if and only if the weighted tree $(T,w)$ is precompact;
\item[\textup{(3)}] the space $\overline T_w$ is separable if and only if the tree $T$ is countable.
\end{enumerate}
\end{corollary}



\section{Proof of Theorem~\ref{t:hedgehog}}\label{s:hedgehog}

To prove Theorem~\ref{t:hedgehog}, we need to demonstrate the equivalence of the following conditions for every metric space $X$:
\begin{enumerate}
\item[\textup{(1)}] $X$ is a Lipschitz image of an injective metric space;
\item[\textup{(2)}] $X$ is a Lipschitz image of some complete $\IR$-tree;
\item[\textup{(3)}] $X$ is Lipschitz connected.
\end{enumerate}

The implication $(2)\Ra(1)$ follows from Proposition~\ref{Rtree} (saying that every complete $\IR$-tree is an injective metric space), and the implication $(1)\Ra(3)$ follows from the Lipschitz connectedness of injective metric spaces, proved in Example~\ref{ex:injective=>Lc}.

To prove that $(3)\Ra(2)$,  assume that a metric space $X$ is Lipschitz connected. Fix any point $x_0\in X$ and {for} every point $a\in X$ find a real number $w_a$ and a $1$-Lipschitz path $\gamma_a\colon [0,w_a]\to X$ such that $\gamma_a(0)=x_0$ and $\gamma_a(w_a)=a$.  If $a=x_0$, then will assume that $w_a=0$. 

Consider the  tree $T\defeq X$ endowed with the partial order $$\preceq:=\{(x,y)\in X\times X:x=y\;\vee\;x=x_0\}$$ The function $w:X\to [0,\infty)$, $w:a\mapsto w_a$, turns $T$ into a weighted tree. The weighted tree $(T,w)$ generates the $\IR$-tree $T_w$. Since the tree $T$ contains no infinite branches, the $\IR$-tree $T_w$ coincides with its completion $\overline T_w=T_w\cup\partial T_w$ and hence $T_w$ is a complete $\IR$-tree.

The $1$-Lipschitz maps $\gamma_a$, $a\in X$, compose a surjective $1$-Lipschitz map $\gamma:T_w\to X$ such that $\gamma ((r_T,0))=x_0$ and $\gamma((a,t))=\gamma_a(t)$ for $(a,t)\in \bigcup_{a\in X}\{a\}\times(0,w_a]$. Indeed for any $(a,t),(b,s)\in T_w$ if $a=b$
	$$d_X(\gamma((a,t)),\gamma((b,s)))=d_X(\gamma_a(t),\gamma_a(s))\leq |t-s|,$$
	because $\gamma_a$ is a 1-Lipschitz map. If $a\ne b$
	$$
	\begin{aligned}
		d_X(\gamma((a,t)),\gamma((b,s)))&=d_X(\gamma_a(t),\gamma_b(s))\leq d_X(\gamma_a(t),x_0)+d_X(x_0,\gamma_b(s))=\\ &=d_X(\gamma_a(t),\gamma_a(0))+d_X(\gamma_b(0),\gamma_b(s))\le t+s,
	\end{aligned}
	$$
	because $\gamma_a$, $\gamma_b$ are 1-Lipschitz maps. The $1$-Lipschitz map $\gamma$ witnesses that $X$ is a $1$-Lipschitz image of the complete $\IR$-tree $T_w$.

\section{Proof of Theorem~\ref{t:compact}}\label{s:compact}

Theorem~\ref{t:compact} will follow as soon as we check the equivalence of the following conditions for every metric space $X$:
\begin{enumerate}
\item[\textup{(1)}] $X$ is a Lipschitz image of  some compact $\IR$-tree;
\item[\textup{(2)}] $X$ is a Lipschitz image of a totally bounded injective metric space;
\item[\textup{(3)}] $X$ is a Lipschitz image of a compact $1$-Lipschitz connected metric space;
\item[\textup{(4)}] $X$ is a Lipschitz image of an $L$-Lipschitz connected compact metric space for some $L\ge 1$;
\item[\textup{(5)}] $X$ is compact, Lipschitz connected, and the intrinsic metric $d_{IX}$  is totally bounded.
\end{enumerate}

The implication $(1)\Ra(2)$ follows from Proposition~\ref{Rtree} (saying that every complete $\IR$-tree is an injective metric space), and the total boundedness of compact metric spaces.
\smallskip

The implication $(2)\Ra(3)$ follows from the completeness and $1$-Lipschitz connectedness of injective metric spaces, and the implication $(3)\Ra(4)$ is obvious.
\smallskip

$(4)\Ra(5)$ Assume that a metric space $(X,d_X)$ is  the image of some $L$-Lipschitz connected compact metric space $(Y,d_Y)$ under a Lipschitz map $f:Y\to X$. Then $X$ is compact and Lipschitz connected, so the intrinsic metric $d_{IX}$ on $X$ is well-defined.
Scaling the metric of the metric space $Y$, we can assume that the Lipschitz map $f$ is $1$-Lipschitz. In this case the definition of the intrinsic metric implies that the function $f$ remains $1$-Lipschitz with respect to the intrinsic metrics on the metric spaces $X$ and $Y$. The $L$-Lipschitz connectedness of the metric space $Y$ implies that $d_Y\le d_{IY}\le L\cdot d_Y$, which implies that the identity map $(Y,d_Y)\to (Y,d_{IY})$ is Lipschitz and hence the metric space $(Y,d_{IY})$ is totally bounded and so is its $1$-Lipschitz image $(X,d_{IX})$.
\smallskip

$(5)\Ra(1)$ Assume that the metric space $(X,d_X)$ is compact, Lipschitz connected, and its intrinsic metric $d_{IX}$ is totally bounded. Let $D\defeq\sup\{d_{IX}(x,y):x,y\in X\}$ be the diameter of the (totally) bounded metric space $(X,d_{IX})$.

Since the metric space $(X,d_{IX})$ is totally bounded, there exists a sequence $(\F_n)_{n\in\w}$ of finite partitions 
of $X$ such that $\F_0=\{X\}$ and for every $n\in\w$, every set $F\in\F_n$  has $d_{IX}$-diameter $\sup\{d_{IX}(x,y):x,y\in F\}<\frac{D+1}{2^{n}}$. Replacing each partition $\F_{n}$ with $n\in\IN$ by the partition $\F_n\wedge \F_{n-1}\defeq\{A\cap B:A\in\F_{n-1},\;\;B\in\F_n,\;\;A\cap B\ne\varnothing\}$, we can assume that every set $F\in\F_n$ is contained in a unique set  $\check F\in\F_{n-1}$. 
For every $n\in\w$ and $F\in\F_n$, choose a point $c_F\in F$.

Consider the tree $T\defeq \bigcup_{n\in\w}(\{n\}\times\F_n)$ endowed with the partial order $\preceq$, defined by $(n,F)\preceq (m,E)$ iff $n\le m$ and $E\subseteq F$. 
Next, endow the tree $T$ with the weight function $w:T\to[0,\infty)$ defined by 
$$
w(n,F)=w_{(n,F)}=\begin{cases}
\frac{D+1}{2^{n-1}}&\mbox{ if $n>0$};\\
0&\mbox{if $n=0$}.
\end{cases}
$$
\begin{figure}[h]
		\includegraphics[scale=0.3]{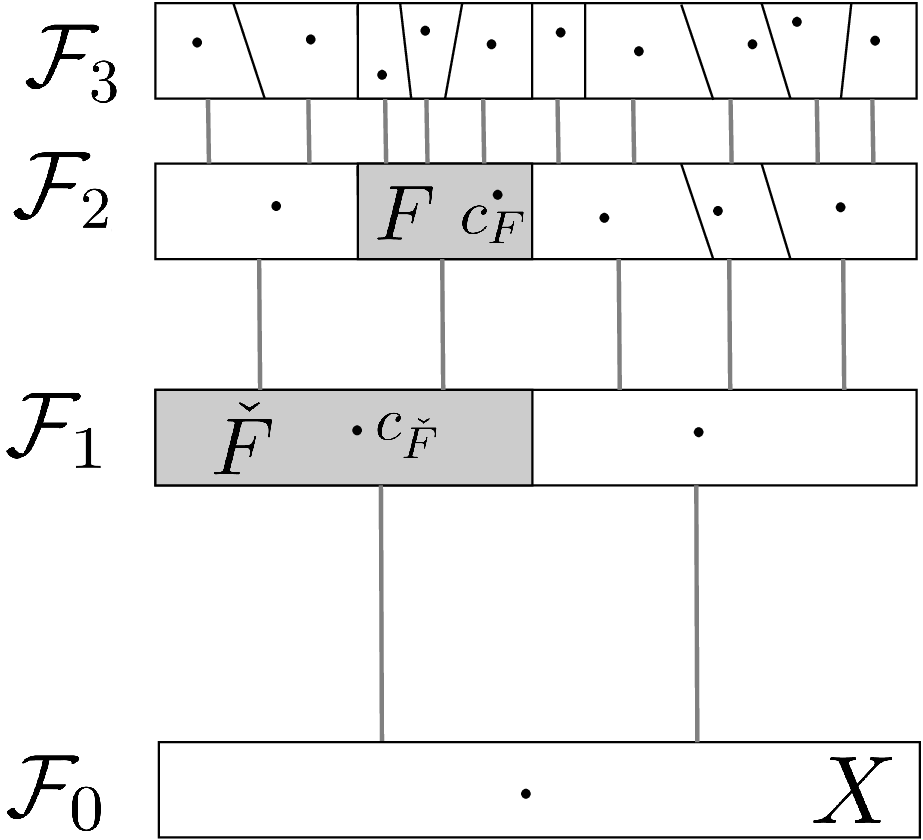}
	\caption{The weighted tree $(T,w)$. }
	\end{figure}

Let $T_w$ be the $\IR$-tree of the weighted tree $(T,w)$. Since the families $\mathcal{F}_n$ are finite and $\sum_{t\in b}w_t = 2(D+1)$ for every $b\in\partial T$, the weighted tree $(T,w)$ is precompact. By Proposition~\ref{calzup}, the metric space $T_w$ is totally bounded and hence its completion $\overline T_w$ is a compact $\IR$-tree.

It remains to construct a $1$-Lipschitz surjective map $\overline\gamma:\overline T_w\to X$. For every $n\in\IN$ and $F\in\F_n$, choose a $1$-Lipschitz function $\gamma_{n,F}:[0,w_{(n,F)}]\to X$ such that $\gamma_{(n,F)}(0)=c_{\check F}$ and $\gamma_{(n,F)}(w_{(n,F)})=c_F$. The function $\gamma_{(n,F)}$ exists because $\{c_{\check F},c_F\}\subseteq \check F\cup F=\check F$ and the set $\check F\in\F_{n-1}$ has $d_{IX}$-diameter $<\frac{D+1}{2^{n-1}}=w_{(n,F)}$. Next, consider the function
$\gamma:T_w\to X$ defined by 
$\gamma((0,X),0)=c_X$ and $\gamma((n,F),t)=\gamma_{n,F}(t)$ for every $n\in\IN$, $F\in\F_n$ and $t\in(0,w_{(n,F)}]$.
The definition of the metric $d_w$ on the $\IR$-tree $T_w$ ensures that the map $\gamma:T_w\to X$ is $1$-Lipschitz. Observe that the image $\gamma[T_w]$ is dense in $X$ (because it contains the dense subset $\{c_F:n\in\IN,\;F\in\F_n\}$ of $X$).

Since the metric space $X$ is complete, the $1$-Lipschitz map $\gamma:T_w\to X$ admits a unique $1$-Lipschitz extension $\overline{\gamma}:\overline T_w\to X$. The total boundedness of the $\IR$-tree $T_w$ implies the compactness of its completion $\overline T_w$. Then the image $\overline\gamma[\overline T_w]$ is a compact dense subset of $X$, witnessing that the $1$-Lipschitz map $\overline \gamma:\overline T_w\to X$ is surjective and hence $X$ is a $1$-Lipschitz image of the compact $\IR$-tree $\overline T_w$.
 
\begin{remark}
Although compactness of $(X,d_X)$ follows from any of the conditions (1)--(5), the compactness of $(X,d_X)$ is not sufficient to infer these conditions, which is showed in Example \ref{e: d vs dp}.
\end{remark}

\section{Proof of Theorem~\ref{t:analytic}}\label{s:separable}

Theorem~\ref{t:analytic} will follow as soon as we check the equivalence of the following conditions for every metric space $X$:
\begin{enumerate}
\item[\textup{(1)}] $X$ is a Lipschitz image of  some separable complete $\IR$-tree;
\item[\textup{(2)}] $X$ is a Lipschitz image of a separable injective metric space;
\item[\textup{(3)}] $X$ is a $1$-Lipschitz image of the Urysohn {universal} metric space; 
\item[\textup{(4)}] $X$ is a $1$-Lipschitz image of a $1$-Lipschitz connected separable complete metric space;
\item[\textup{(5)}] $X$ is a Lipschitz image of an $L$-Lipschitz connected analytic metric space for some $L\ge 1$;
\item[\textup{(6)}] $X$ is analytic, Lipschitz connected, and the intrinsic metric $d_{IX}$ is separable.
\end{enumerate}

The implication $(1)\Ra(2)$ follows from Proposition~\ref{Rtree} (saying that every complete $\IR$-tree is an injective metric space).
\smallskip

To prove that $(2)\Ra(3)$, assume that the metric space $X$ is the image of a separable injective space $Y$ under a Lipschitz map $f:Y\to X$. Scaling the metric of the metric space $Y$, we can assume that the Lipschitz map $f$ is $1$-Lipschitz. By \cite{Urys}, the Urysohn universal metric space $\mathbb U$ contains an isometric copy of any separable metric space. So, there exists an isometric embedding $e:Y\to\mathbb U$. By the injectivity of $Y$, the $1$-Lipschitz map $e^{-1}:e[Y]\to Y$ defined on the subset $e[Y]$ of $\mathbb U$ admits a $1$-Lipschitz extension $g:\mathbb U\to Y$. Then $f\circ g:\mathbb U\to X$ is a surjective $1$-Lipschitz map, witnessing that $X$ is a $1$-Lipschitz image of the Urysohn universal metric space $\mathbb U$.
\smallskip

To see the implication $(3)\Ra(4)$ it suffices to check that the  Urysohn universal metric space $\mathbb U$ is $1$-Lipschitz connected. Given any points $x,y\in \mathbb U$, consider the function $f:\{0,d(x,y)\}\to\{x,y\}\subseteq\mathbb U$ such that $f(0)=x$ and $f(d(x,y))=y$. By \cite{Bo, H} {(see also \cite{BS, KR, MJ})}, the isometric embedding $f:\{0,d(x,y)\}\to\mathbb U$ extends to a $1$-Lipschitz map $f:[0,d(x,y)]\to\mathbb U$, witnessing that $\mathbb U$ is $1$-Lipschitz connected.
\smallskip

The implication $(4)\Ra(5)$ is trivial.
\smallskip

$(5)\Ra(6)$ Assume that for some $L\ge 1$, a metric space $(X,d_X)$ is  the image of some $L$-Lipschitz connected analytic metric space $(Y,d_Y)$ under a Lipschitz map $f:Y\to X$. Then $X$ is analytic and Lipschitz connected, so the intrinsic metric $d_{IX}$ on $X$ is well-defined.
Scaling the metric of the metric space $Y$, we can assume that the Lipschitz map $f$ is $1$-Lipschitz. In this case the definition of the intrinsic metric implies that the function $f$ remains $1$-Lipschitz with respect to the intrinsic metrics on the metric spaces $X$ and $Y$. The $L$-Lipschitz connectedness of the metric space $Y$ implies that $d_Y\le d_{IY}\le L\cdot d_Y$, which implies that the identity map $(Y,d_Y)\to (Y,d_{IY})$ is Lipschitz and hence the metric space $(Y,d_{IY})$ is separable and so is its $1$-Lipschitz image $(X,d_{IX})$.
\smallskip

$(6)\Ra(1)$ Assume that the metric space $(X,d_X)$ is analytic, Lipschitz connected, and its intrinsic metric $d_{IX}$ is separable. By Proposition~\ref{d vs dp}, the metric space $X_I\defeq (X,d_{IX})$ is analytic and hence it is the image of a separable complete metric space $(Z,d_Z)$ under a  surjective continuous map $f:Z\to X_I$. Consider the metric $\hat d_Z$ on $Z$, defined by the formula $\hat d_Z(z,z')\defeq\max\{d_Z(z,z'),d_{IX}(f(z),f(z'))\}$. The continuity of the function $f$ ensures that the metric $\hat d_Z$ is topologically equivalent to the metric $d_Z$. Since $d_Z\le\hat d_Z$, the completeness of the metric $d_Z$ implies the completeness of the metric $\hat d_Z$. Replacing the metric $d_Z$ by the metric $\hat d_Z$, we can assume that the function $f:Z\to X_I$ is $1$-Lipschitz. 

Since the metric space $Z$ is separable, there exists a sequence $(\F_n)_{n\in\w}$ of countable partitions of $Z$ such that $\F_0=\{Z\}$ and for every $n\in\IN$, every set $F\in\F_n$ has diameter $\sup\{d_Z(x,y):x,y\in F\}<\frac{1}{2^{n}}$. Replacing each partition $\F_{n}$ by the partition $\F_n\wedge \F_{n-1}\defeq\{A\cap B:A\in\F_{n-1},\;\;B\in\F_n,\;\;A\cap B\ne\varnothing\}$, we can assume that every set $F\in\F_n$ is contained in a unique set $\check F\in\F_{n-1}$. 
For every $n\in\w$ and $F\in\F_n$, choose a point $c_F\in F$.

Consider the countable tree $T\defeq \bigcup_{n\in\w}(\{n\}\times\F_n)$ endowed with the partial order $\preceq$, defined by $(n,F)\preceq (m,E)$ iff $n\le m$ and $E\subseteq F$.  Next, endow the tree $T$ with the weight function $w:T\to[0,\infty)$ defined by $w_{(0,Z)}=0$ and $w_{(n,F)}=d_{IX}(f(c_{\check F}),f(c_F))+\tfrac1{2^n}$ for every $n\in\IN$ and $F\in\mathcal F_n$ (in this formula $\check F$ denotes the unique set in $\F_{n-1}$ that contains the set $F$).

Let $T_w$ be the $\IR$-tree of the weighted tree $(T,w)$ and $\overline T_w=T_w\cup\partial T_w$ be its completion. Proposition~\ref{przel} ensures that the $\IR$-tree $T_w$ is separable and hence its completion $\overline T_w$ is a separable complete $\IR$-tree.

It remains to construct a $1$-Lipschitz surjective map $\overline\gamma:\overline T_w\to X$. For every $n\in\IN$ and $F\in\F_n$, choose a $1$-Lipschitz function $\gamma_{n,F}:[0,w_{(n,F)}]\to X$ such that $\gamma_{n,F}(0)=f(c_{\check F})$ and $\gamma_{n,F}(w_{(n,F)})=f(c_F)$. The function $\gamma_{n,F}$ exists because $d_{IX}(f(c_{\check F}),f(c_F))<w_{(n,F)}$. Next, consider the function
$\gamma:T_w\to X$ defined by 
$\gamma((0,X),0)=f(c_{Z})$ and $\gamma((n,F),t)=\gamma_{n,F}(t)$ for every $n\in\IN$, $F\in\F_n$ and $t\in(0,w_{(n,F)}]$.
The definition of the metric $d_w$ on the $\IR$-tree $T_w$ ensures that the map $\gamma:T_w\to X$ is $1$-Lipschitz. 

Consider the set $\partial T$ of all infinite branches of the tree $T$ and observe that every branch $b\in\partial T$ is equal to the set $\{(n,F_n):n\in\w\}$ for some decreasing sequence $(F_n)_{n\in\w}\in\prod_{n\in\w}\F_n$. 
Observe that for every $n\ge 2$, the set $\{c_{\check F_n},c_{F_n}\}$ is a subset of the set $\check F_n\in\mathcal F_{n-1}$ and hence $d_Z(c_{F_{n-1}},c_{F_n})=d_Z(c_{\check F_n},c_{F_n})<\frac1{2^{n-1}}$.
Then the sequence $(c_{F_n})_{n\in\w}$ is Cauchy in the complete metric space $Z$ and has a limit $\lambda(b)\defeq\lim_{n\to\infty}c_{F_n}\in Z$, which coincides with the unique point of the  intersection $\bigcap_{n\in\w}\overline{F_n}$.  
The correspondence $\lambda:\partial T\mapsto\lambda(b)$, determines a function $\lambda:\partial T\to Z$. Observe that for every point $z\in Z$ there exists a unique branch $b=\{(n,F_n)\}_{n\in\w}\in\partial T$ such that $z\in\bigcap_{n\in\w}F_n=\bigcap_{n\in\w}\overline F_n=\{\lambda(b)\}$, witnessing that the function $\lambda:\partial T\to Z$ is surjective. Taking into account that diameters of the sets in the partition $\F_n$ tend to zero as $n$ tends to infinity, one can show that the map $\lambda:\partial T\to Z$ is continuous with respect to the topology on $\partial T$ inherited from the Tychonoff product $\prod_{n\in\w}(\{n\}\times\F_n)$ of discrete spaces $\{n\}\times\F_n$, $n\in\w$. 

Observe that for every branch $b=\{(n,F_n)\}_{n\in\w}$ of the tree $T$ and every $n\in\IN$ we have {$w_{(n,F_n)}=d_{IX}(f(c_{F_{n-1}}),f(c_{F_n}))+\frac1{2^n}\le d_Z(c_{F_{n-1}},c_{F_n})+\frac1{2^n}<\frac1{2^{n-1}}+\tfrac1{2^n}$} and hence $\sum_{t\in b}w_t=\sum_{n\in\w}w_{(n,F_n)}<\infty$, witnessing that $\partial T_w=\partial T$.

Since the map $\gamma:T_w\to X$ is $1$-Lipschitz, there exists a unique $1$-Lipschitz extension $\overline \gamma:\overline T_w\to\overline X$ to the completion of the space $X$. We claim that $\overline\gamma[\overline T_w]=X$. 

The inclusion $\overline\gamma[\overline T_w]\subseteq X$ will follow as soon as we check that $\overline\gamma[\partial T]\subseteq X$. Fix any element $b\in\partial T$ and find a unique decreasing sequence of sets $(F_n)_{n\in\w}\in\prod_{n\in\w}\F_n$ such that $b=\{(n,F_n)\}_{n\in\w}$.
Then $\lambda(b)=\lim_{n\to\infty}c_{F_n}$. 
 The definition of the metric on the $\IR$-tree $\overline T_w$ ensures that $$b=\lim_{n\to\infty}((n,F_n),w_{(n,F_n)})$$ in $\overline T_w$ and hence $\overline \gamma(b)=\lim_{n\to\infty}\gamma((n,F_n),w_{(n,F_n)})=\lim_{n\to\infty}\gamma_{n,F}(w_{(n,F_n)})=\lim_{n\to\infty}f(c_{F_n})=f(\lim_{n\to\infty}c_{F_n})=f(\lambda(b))\in f[Z]=X$.
 
 Next, we show that $X\subseteq\overline\gamma[\overline T_w]$. Given any point $x\in X$, choose any point  $z\in f^{-1}(x)\subseteq Z$ and find a unique branch $b=\{(n,F_n)\}_{n\in\w}\in\partial T$ such that $z\in\bigcap_{n\in\w}F_n$. Then $z=\lambda(b)$ and $x=f(\lambda(b))=\overline\gamma(b)$. 
 
Therefore, $\overline \gamma:\overline T_w\to X$ is a $1$-Lipschitz surjective map witnessing that $X$ is a $1$-Lipschitz image of the separable complete $\IR$-tree $\overline T_w$. 

\begin{remark}
Although analyticity and separability of $(X,d_X)$ follows from each of the conditions (1)--(6), the analyticity of $(X,d_X)$ is not sufficient to infer the equivalent conditions of Theorem \ref{t:analytic}, which is showed in Example \ref{e: d vs dp}.
\end{remark}

\section*{Acknowledgement} This work was co-financed by the Minister of Science (Poland) under the ``Regional Excellence Initiative'' program (project no.: RID/SP/0015/2024/01).


\end{document}